\documentclass[12pt,reqno]{amsart}
\usepackage{amsfonts}

\usepackage{amssymb,amsmath,graphicx,amsfonts,euscript}
\usepackage{color}

\newtheorem{thm}{Theorem}[section]

\newtheorem{cor}{Corollary}[section]

\newtheorem{lemma}{Lemma}[section]

\theoremstyle{definition}
\newtheorem{define}{Definition}[section]

\theoremstyle{remark}
\newtheorem{rem}{Remark}[section]

\allowdisplaybreaks \voffset=-0.2in \numberwithin{equation}{section}

\begin{document}
\bigskip

\centerline{\Large\bf  Global regularity of 2D tropical climate model  }
\smallskip
\centerline{\Large\bf with zero thermal diffusion}
\bigskip

\centerline{Zhuan Ye}

\medskip

\centerline{Department of Mathematics and Statistics, Jiangsu Normal University,
}
\medskip

\centerline{101 Shanghai Road, Xuzhou 221116, Jiangsu, PR China}

\medskip

\centerline{E-mail: \texttt{yezhuan815@126.com
}}

\bigskip

{\bf Abstract:}~~%
This article studies the global regularity problem of the two-dimensional zero thermal diffusion tropical climate model with fractional dissipation, given by $(-\Delta)^{\alpha}u$ in the barotropic mode equation and by $(-\Delta)^{\beta}v$ in the first baroclinic mode of the vector velocity equation. More precisely, we show that the global regularity result holds true as long as $\alpha+\beta\geq2$ with $1<\alpha<2$. In addition, with no dissipation from both the temperature and the first baroclinic mode of the vector velocity, we also establish the global regularity result with the dissipation strength at the logarithmically supercritical level. Finally, our arguments can be extended to obtain the corresponding global regularity results of the higher dimensional cases.
{\vskip 1mm
 {\bf AMS Subject Classification 2010:}\quad 35Q35; 35B65; 76D03.

 {\bf Keywords:}
Tropical climate model; Global regularity; Fractional dissipation.}

\vskip .4in
\section{Introduction}
In this paper, we consider the global regularity problem of the
following two-dimensional (2D) tropical climate model with zero thermal diffusion
\begin{equation}\label{casTCM}
\left\{\aligned
&\partial_{t}u+(u \cdot \nabla) u + (-\Delta)^{\alpha}u+\nabla p+\nabla\cdot(v\otimes v)=0,\,\,\,\,\,x\in \mathbb{R}^{2},\,\,t>0, \\
&\partial_{t}v+(u \cdot \nabla)v+(-\Delta)^{\beta}v+\nabla \theta+ (v \cdot \nabla)u=0,\\
&\partial_{t}\theta+(u \cdot \nabla) \theta+\nabla\cdot v=0,\\
&\nabla\cdot u=0,\\
&u(x, 0)=u_{0}(x),  \quad v(x,0)=v_{0}(x),  \quad \theta(x,0)=\theta_{0}(x),
\endaligned\right.
\end{equation}
where $u=(u_{1}(x,t),\,u_{2}(x,t))$ is the barotropic mode, $v=(v_{1}(x,t),\,v_{2}(x,t))$ is the first baroclinic
mode of the vector velocity, $p=p(x,t)$
is the scalar pressure and $\theta=\theta(x,t)$ is the scalar temperature, respectively. Here $v\otimes v$ denotes the tensor product, namely  $v\otimes v = (v_{i}v_{j})$, $\alpha\geq 0$ and $\beta\geq0$ are real parameters. The fractional Laplacian operator $(-\Delta)^{\gamma}$ with $\gamma>0$ is defined through the Fourier transform, namely
$$\widehat{(-\Delta)^{\gamma}f}(\xi)=|\xi|^{2\gamma}\widehat{f}(\xi).$$
Recently, a great attention has been devoted to the study of nonlocal problems driven by fractional Laplacian type operators in the literature, not only for a pure academic interest, but also for the various applications in different fields, for example the physical phenomena in hydrodynamics, molecular biology such as anomalous diffusion in semiconductor growth, probability, finance and so on (see, e.g., \cite{Dronioui,MeMMt,Pekalskis,Woyczy}). We remark the convention that by
$\alpha=0$ we mean that there is no dissipation in $\eqref{casTCM}_{1}$, and similarly $\beta=0$ represents that there is no dissipation in $\eqref{casTCM}_{2}$.

\vskip .1in
The inviscid version of the model (\ref{casTCM}), namely $\alpha=\beta=0$, was
derived by Frierson, Majda and Pauluis for large-scale dynamics of precipitation fronts in the tropical atmosphere \cite{FMP}.
Mathematically these fractional dissipation terms of (\ref{casTCM}) increase the regularity for the system, which also significantly change the model properties and physics. The original tropical climate model
is derived from the inviscid primitive equations \cite{FMP},
and its viscous counterpart with the standard Laplacian can be derived by the same argument from the viscous primitive
equations (see \cite{LiTiti2}). For more background on the tropical climate model, we refer to \cite{G80,MA03,Mt66} and the references therein. The model (\ref{casTCM}) with fractional diffusion operators are relevant in modeling the so-called anomalous diffusion. Moreover, the model (\ref{casTCM}) allows us to investigate the long-range diffusive interactions.

\vskip .1in
Let us review some previous works on the tropical climate model. For the case \eqref{casTCM} with $\alpha=\beta=1$, Li and Titi \cite{LiTiti} (also Dong, Wu and Ye \cite{DWyeJNS18}) established the global regularity result by introducing a combined quantity of $v$ and $\theta$.
The global regularity for (\ref{casTCM}) when $\alpha>0$, $\beta=1$ and the equation of $\theta$ contains $\Delta \theta$ was obtained in \cite{Ye16jmaa}. Later, Dong, Wang, Wu and Zhang \cite{DWWZ} proved the global regularity for \eqref{casTCM} when $\alpha+\beta=2$ and $1<\beta\leq\frac{3}{2}$ or $\alpha=2,\,\beta=0$. By introducing several combined quantities and by exploiting the De Giorgi-Nash type estimate for for the transport-diffusion equation
with a more general forcing term, this result was further improved in \cite{DWye2018}. More precisely, $\beta\geq\frac{3-\alpha}{2}$ with $0<\alpha<1$ would ensure the global regularity for the model \eqref{casTCM}. We point out that when $\beta>\frac{3}{2}$, the global regularity also remains valid (see \cite{DWWYez17}). The global regularity results for many other cases were also proven in \cite{DWWYez17,DWyeJNS18}. It should be noted that all the above mentioned works require the restriction $\beta\geq1$ and the remainder case $\beta<1$ remains unsolved. In fact, this is the aim of the present paper. Now let us state our results. The first main result considering the case $\alpha+\beta\geq2$ with $1<\alpha<2$ can be stated as follows.
\vskip .1in
\begin{thm}\label{Th1}
 Consider (\ref{casTCM}) with $\alpha$ and $\beta$ satisfying
\begin{equation}\label{asfghr}
 \alpha+\beta\geq2,\quad 1<\alpha<2.
\end{equation}
Assume the initial data $(u_{0}, v_{0}, \theta_{0})
\in H^{s}(\mathbb{R}^{2})\times H^{s}(\mathbb{R}^{2})\times H^{s}(\mathbb{R}^{2})$ with $s>2$, and $\nabla\cdot u_0=0$. Then (\ref{casTCM}) admits a unique global solution $(u, v, \theta)$ such that for any
given $T>0$,
$$u\in C([0, T]; H^{s}(\mathbb{R}^{2}))\cap L^{2}(0, T; H^{s+\alpha}(\mathbb{R}^{2})),$$
$$v \in C([0, T]; H^{s}(\mathbb{R}^{2}))\cap L^{2}(0, T;
H^{s+\beta}(\mathbb{R}^{2})),\qquad \theta\in C([0, T]; H^{s}(\mathbb{R}^{2})).$$
\end{thm}

\vskip .1in
For the borderline case $\alpha=2,\,\beta=0$, we have the following global regularity result with logarithmically supercritical dissipation.
\begin{thm}\label{Th2}
Consider \eqref{casTCM} with $\alpha=2,\,\beta=0$, namely,
\begin{equation}\label{logcasTCM}
\left\{\aligned
&\partial_{t}u+(u \cdot \nabla) u +\mathcal{L}^{2} u+\nabla p+\nabla\cdot(v\otimes v)=0,\,\,\,\,\,x\in \mathbb{R}^{2},\,\,t>0, \\
&\partial_{t}v+(u \cdot \nabla)v+\nabla \theta+ (v \cdot \nabla)u=0,\\
&\partial_{t}\theta+(u \cdot \nabla) \theta+\nabla\cdot v=0,\\
&\nabla\cdot u=0,\\
&u(x, 0)=u_{0}(x),  \quad v(x,0)=v_{0}(x),  \quad \theta(x,0)=\theta_{0}(x),
\endaligned\right.
\end{equation}
where the operator $\mathcal{L}$ is defined by
$$\widehat{\mathcal{L}u}(\xi)=\frac{|\xi|^{2}}{g(\xi)}\widehat{u}(\xi)$$
for some non-decreasing symmetric function $g(\tau)\geq 1$ defined on $\tau\geq 0$. Assume the initial data $(u_{0}, v_{0}, \theta_{0})
\in H^{s}(\mathbb{R}^{2})\times H^{s}(\mathbb{R}^{2})\times H^{s}(\mathbb{R}^{2})$ with $s>2$, and $\nabla\cdot u_0=0$. If $g$ satisfies the following growth condition
\begin{equation}\label{logcobd}
\int_{e}^{\infty}\frac{d\tau}{\tau\sqrt{\ln \tau} g(\tau)}=\infty,\end{equation}
then the system \eqref{logcasTCM} admits a unique global solution $(u, v, \theta)$ such that for any
given $T>0$,
$$(u, v, \theta) \in C([0, T]; H^{s}(\mathbb{R}^{2})),\quad \mathcal{L} u\in L^{2}([0, T]; H^{s}(\mathbb{R}^{2})).
$$
\end{thm}

\begin{rem}  \rm
It is worthwhile to mention that when $\alpha+\beta\geq2$ with $1<\alpha<2$, the following 2D generalized magnetohydrodynamic (MHD)
equations admit a unique global regular solution (see \cite{wujmfm})
\begin{equation}\label{mhd}
\left\{\aligned
&\partial_{t}u+(u \cdot \nabla) u +(-\Delta)^{\alpha} u+\nabla p- \nabla \cdot (b\otimes b)=0,\,\,\,\,\,x\in \mathbb{R}^{2},\,\,t>0, \\
&\partial_{t} b + (u \cdot \nabla)b+(-\Delta)^{\beta} b -  (b \cdot \nabla)u=0,\\
&\nabla\cdot u=0, \quad \nabla\cdot b=0.
\endaligned\right.
\end{equation}
On the one hand, the tropical climate model (\ref{casTCM}) bears some similarities. Obviously, when $\theta$ is a constant and $\nabla\cdot v=0$, the system (\ref{casTCM}) reduces to the 2D MHD-type equations \eqref{mhd}.
On the other hand, compared with \eqref{mhd}, the tropical climate model (\ref{casTCM}) is more involved to deal with. One key point is due to the presence of $\nabla\theta$ in the $v$-equation. The second point lies on that the $\theta$ satisfies a pure transport equation with a forcing term $-\nabla\cdot v$.
Another reason is that $b$ can be assumed divergence-free due
to the fact that this property is preserved as time evolves, but $v$ in
(\ref{casTCM}) is not divergence-free. This makes the global regularity of (\ref{casTCM}) very complicated when compared with the corresponding MHD equations. It is hoped that this study on the tropical climate
model will help us to further understand the global regularity issue on the MHD equations.
\end{rem}

\vskip .1in
\begin{rem}\rm
We also remark that the typical examples satisfying the condition (\ref{logcobd}) are as follows
\begin{equation}
\begin{split}
& g(\xi)=\big[\ln(e+|\xi|)\big]^{\frac{1}{2}}; \\
& g(\xi)=\big[\ln(e+|\xi|)\big]^{\frac{1}{2}} \ln (e+\ln(e+|\xi|));\\
& g(\xi)=\big[\ln(e+|\xi|)\big]^{\frac{1}{2}} \ln (e+\ln(e+|\xi|)) \ln(e+\ln(e+\ln(e+|\xi|))).\nonumber
\end{split}
\end{equation}
Consequently, Theorem \ref{Th2} improves \cite[Theorem 1.2]{DWWZ} logarithmically.
\end{rem}

\vskip .1in
\begin{rem}\rm
The global regularity for the model (\ref{casTCM}) when $\alpha+\beta<2$ with $1<\alpha<2$ is currently open. For this case, it appears extremely difficult to establish the global $H^{\varrho}$-bound even for small $\varrho>0$. As stated above, when $\nabla\cdot v=0$ and $\theta$ is a constant, (\ref{casTCM}) reduces to the 2D MHD-type equations \eqref{mhd}, whose global regularity result still requires dissipation only logarithmically weaker than the dissipation level  $\alpha+\beta=2$ with $1<\alpha<2$ (see \cite{wujmfm} for details). Therefore, this is a very interesting and changeling problem, which is left for future.
\end{rem}

\vskip .1in
We point out that the corresponding global regularity results of (\ref{casTCM}) are true for the higher dimensions. More precisely, we have the following corollaries.
\begin{cor}\label{Th3}
Consider the following $n$-dimensional ($n\geq3$) incompressible tropical climate model
\begin{equation}\label{awer1}
\left\{\aligned
&\partial_{t}u+(u \cdot \nabla) u +(-\Delta)^{\alpha}u+\nabla p+\nabla\cdot(v\otimes v)=0,\,\,\,\,\,x\in \mathbb{R}^{n},\,\,t>0, \\
&\partial_{t}v+(u \cdot \nabla)v+(-\Delta)^{\beta}v+\nabla \theta+ (v \cdot \nabla)u=0,\\
&\partial_{t}\theta+(u \cdot \nabla) \theta+\nabla\cdot v=0,\\
&\nabla\cdot u=0,\\
&u(x, 0)=u_{0}(x),  \quad v(x,0)=v_{0}(x),  \quad \theta(x,0)=\theta_{0}(x).
\endaligned\right.
\end{equation}
Assume the initial data $(u_{0}, v_{0}, \theta_{0})
\in H^{s}(\mathbb{R}^{n})\times H^{s}(\mathbb{R}^{n})\times H^{s}(\mathbb{R}^{n})$ with $s>1+\frac{n}{2}$, and $\nabla\cdot u_0=0$.
If $\alpha$ and $\beta$ satisfy
\begin{equation}
 \alpha+\beta\geq1+\frac{n}{2},\quad \frac{1}{2}+\frac{n}{4}\leq\alpha<1+\frac{n}{2},\nonumber
\end{equation}
then (\ref{awer1}) admits a unique global solution $(u, v, \theta)$ such that for any given $T>0$,
$$u\in C([0, T]; H^{s}(\mathbb{R}^{n}))\cap L^{2}(0, T; H^{s+\alpha}(\mathbb{R}^{n})),$$
$$v \in C([0, T]; H^{s}(\mathbb{R}^{n}))\cap L^{2}(0, T;
H^{s+\beta}(\mathbb{R}^{n})),\qquad \theta\in C([0, T]; H^{s}(\mathbb{R}^{n})).$$
\end{cor}

\begin{cor}\label{Th4}
Consider the following $n$-dimensional ($n\geq3$) incompressible tropical climate model
\begin{equation}\label{awer2}
\left\{\aligned
&\partial_{t}u+(u \cdot \nabla) u +\mathcal{L}^{2} u+\nabla p+\nabla\cdot(v\otimes v)=0,\,\,\,\,\,x\in \mathbb{R}^{n},\,\,t>0, \\
&\partial_{t}v+(u \cdot \nabla)v+\nabla \theta+ (v \cdot \nabla)u=0,\\
&\partial_{t}\theta+(u \cdot \nabla) \theta+\nabla\cdot v=0,\\
&\nabla\cdot u=0,\\
&u(x, 0)=u_{0}(x),  \quad v(x,0)=v_{0}(x),  \quad \theta(x,0)=\theta_{0}(x),
\endaligned\right.
\end{equation}
where the operator $\mathcal{L}$ is defined by
$$\widehat{\mathcal{L}u}(\xi)=\frac{|\xi|^{1+\frac{n}{2}}}{g(\xi)}\widehat{u}(\xi)$$
for some non-decreasing symmetric function $g(\tau)\geq 1$ defined on $\tau\geq 0$. Assume the initial data $(u_{0}, v_{0}, \theta_{0})
\in H^{s}(\mathbb{R}^{n})\times H^{s}(\mathbb{R}^{n})\times H^{s}(\mathbb{R}^{n})$ with $s>1+\frac{n}{2}$, and $\nabla\cdot u_0=0$. If $g$ satisfies the following growth condition
\begin{equation}
\int_{e}^{\infty}\frac{d\tau}{\tau\sqrt{\ln \tau} g(\tau)}=\infty,\nonumber
\end{equation}
then the system \eqref{awer2} admits a unique global solution $(u, v, \theta)$ such that for any
given $T>0$,
$$(u, v, \theta) \in C([0, T]; H^{s}(\mathbb{R}^{n})),\quad \mathcal{L} u\in L^{2}([0, T]; H^{s}(\mathbb{R}^{n})).
$$
\end{cor}

\begin{rem}\rm
The proof of Corollary \ref{Th3} and Corollary \ref{Th4} can be performed by the same arguments adopted in proving Theorem \ref{Th1} and Theorem \ref{Th2}, respectively. It suffices to make some suitable modifications due to the change of dimensions. To avoid the redundancy, we thus omits the details.
\end{rem}

\vskip .2in
Now we give some rough ideas on our proof of Theorem \ref{Th1} and Theorem \ref{Th2}. The proof is not straightforward and demands new techniques. We describe the main difficulties and explain the techniques to overcome them.
We state that the existence and uniqueness of local smooth solutions can be performed through the standard approach (see for example \cite[Proposition A.1]{DWyeJNS18}). Thus, in order to complete the proof of Theorem \ref{Th1} and Theorem \ref{Th2}, it is sufficient to establish the global {\it a priori} estimates that hold for any given finite time $T>0$.

\vskip .1in
We begin with Theorem \ref{Th1}. The global $L^2$ bound for $(u, v, \theta)$, along with the time integrability of $\|\Lambda^\alpha u\|_{L^2}^2$, $\|\Lambda^\beta v\|_{L^2}^2$, is immediate due to the special structure of
(\ref{casTCM}) and $\nabla\cdot u=0$. The next step is to derive the global $H^1$-bound for $(u, v, \theta)$, but direct energy estimates do not appear to easily achieve this bound. One of the obstacles is that there is no dissipation in the $\theta$-equation and the dissipation in the $v$-equation is not strong enough (the case $\beta<1$ is our main target). As a matter of fact, to derive any regularity of $\theta$, we need to control $\|\nabla u\|_{L^\infty}$ or for a quantity that is close to the regularity level of $\|\nabla u\|_{L^\infty}$. However, to achieve this goal, one has to first obtain the global bound on the forcing in the equation of $u$, namely $\|\nabla\nabla\cdot (v\otimes v)\|_{L^\infty}$.
Unfortunately the equation of $v$ involves $\nabla\theta$ and one has to
know the regularity of $\theta$ first in order to bound $v$. This
tangling makes the estimates of the regularity of $(u, v, \theta)$ very complicated. If one adopts the ideas of \cite{LiTiti,Ye16jmaa,DWWZ,DWye2018,DWWYez17,DWyeJNS18}, then it heavily depends on the restriction $\beta\geq1$. However, for the MHD equations \eqref{mhd}, we do not need to face that problem, and as we know, the global $H^1$-bound for $(u, b)$
can be easily derived by the simple interpolation inequality.
To overcome these difficulties, we make fully use of the key space-time estimate (see Lemma \ref{fraest01}) to the $u$-equation to derive the following crucial bound under the assumption $\alpha+\beta\geq2$ with $1<\alpha<2$ (see Lemma \ref{cale22} for details)
\begin{align}\label{keyt1}
\int_{0}^{t}{
\|\nabla u(\tau)\|_{L^{\infty}}\,d\tau}\leq C(t,u_0, v_0, \theta_0).
\end{align}
With \eqref{keyt1} in hand, we can show the global $H^1$-bound for $(u, v, \theta)$. Then, the global $H^s$-bound for $(u, v, \theta)$ will be obtained.
This ends the proof of Theorem \ref{Th1}.

\vskip .1in
We now explain the proof of Theorem \ref{Th2}. We first have the following basic global $L^2$-energy
\begin{align}\label{exstp14}
\int_{0}^{t}{\Big\|\frac{\Lambda^{2}}{ g (\Lambda)}u \Big\|_{L^{2}}^{2} \,d\tau} \leq
C(u_{0},v_{0},\theta_{0}).
\end{align}
The above bound \eqref{exstp14} plays an important role in deriving the higher regularity of the solution. In general, the next step is to show the global $H^1$-bound for $(u, v, \theta)$. However, one may face the following difficult term due to the absence of $\nabla\cdot v=0$
$$\int_{\mathbb{R}^{2}}v\,\nabla v \,\nabla^{2}u\,dx,$$
which can not be controlled by \eqref{exstp14}. On the one hand, we do not need to
face that problem for the MHD equations \eqref{mhd} due to $\nabla\cdot b=0$. On the other hand, when $\mathcal{L}=-\Delta$, namely $g\equiv1$, one can still derive the global $H^1$-bound for $(u, v, \theta)$ with the help of the redefined Gronwall type inequality (see \cite[Proposition 5]{DWWZ} for details).
To bypass these difficulties coming from $\nabla\cdot v\neq0$ and the logarithmically supercritical dissipation, we resort to establish the lower global regularity for $(u, v, \theta)$, namely, $H^{\sigma}$-bound (for any positive $\sigma<1$). In fact, invoking several commutator estimates yields (see \eqref{tyht305})
\begin{align*}
\frac{d}{dt}X(t)+Y(t)\leq C(\|\nabla u\|_{L^{\infty}}+\|\Delta u\|_{L^{2}})X(t),
\end{align*}
where
$$X(t):=\|\Lambda^{\sigma}u(t)\|_{L^{2}}^{2}+\|\Lambda^{\sigma}v(t)\|_{L^{2}}^{2}
+\|\Lambda^{\sigma}\theta(t)\|_{L^{2}}^{2},\qquad Y(t):=\|\mathcal{L}\Lambda^{\sigma}
u(t)\|_{L^{2}}^{2}.$$
In order to handle the two terms $\|\nabla u\|_{L^{\infty}}$ and $\|\Delta u\|_{L^{2}}$, we take fully exploit of the Littlewood-Paley technique to show that
\begin{align}
 \frac{d}{dt}X(t)+Y(t) \leq C(e+X(t))\sqrt{\ln(e+X(t))}g \left((e+X(t))^{\frac{1}{\kappa}}\right)\left(1+
 \Big\|\frac{\Lambda^{2}}{ g (\Lambda)}u \Big\|_{L^{2}}\right),\nonumber
\end{align}
where $\kappa>0$ is a constant.
Keeping in mind \eqref{exstp14} and \eqref{logcobd}, we infer that
\begin{align}\label{vftyu8}
X(t)+\int_{0}^{t}Y(\tau)\,d\tau\leq C(t,u_{0},v_{0},\theta_{0}).
\end{align}
Finally, with \eqref{vftyu8} in hand, we can propagate all the higher regularities. Consequently, this ends the proof of Theorem \ref{Th2}.

 \vskip .2in
The rest of the paper is organized as follows. In Section 2 we carry out the proof of Theorem \ref{Th1}. Section 3 is devoted to the proof of Theorem \ref{Th2}.
In Appendix \ref{aaa1}, we collect the Littlewood-Paley
decomposition, Besov spaces and some related facts. For the sake of completeness, we present the proof of \eqref{cftghy1} in Appendix \ref{aaa3}.

\vskip .3in
\section{The proof of Theorem \ref{Th1}}\setcounter{equation}{0}
This section is devoted to the proof of Theorem \ref{Th1}.
Before the proof, we will state several notations. For simplicity, we always denote $\Lambda=(-\Delta)^{\frac{1}{2}}$. In this paper, we shall use the convention that $C$ denotes a generic constant, whose value may change from line to line. We shall write $C(\cdot,\cdot\cdot\cdot,\cdot)$ as the constant $C$ depends only on the quantities appearing in parentheses.
For a quasi-Banach space $X$ and for any $0<T\leq\infty$, we use standard notation $L^{p}(0,T;X)$ or $L_{T}^{p}(X)$ for
the quasi-Banach space of Bochner measurable functions $f$ from $(0,T)$ to $X$ endowed with the norm
\begin{equation*}
\|f\|_{L_{T}^{p}(X)}:=\left\{\aligned
&\left(\int_{0}^{T}{\|f(.,t)\|_{X}^{p}\,dt}\right)^{\frac{1}{p}}, \,\,\,\,\,1\leq p<\infty,\\
&\sup_{0\leq t\leq T}\|f(.,t)\|_{X},\qquad\qquad p=\infty.
\endaligned\right.
\end{equation*}

\vskip .1in
Now we begin with the basic global $L^2$-bound.
\begin{lemma}\label{cale21}
Assume $(u_{0},v_{0},\theta_{0})$ satisfies the conditions stated in Theorem \ref{Th1}.
Then for any corresponding smooth solution $(u, v, \theta)$ of  (\ref{casTCM}), we have, for any $t>0$,
\begin{align}
\|u(t)\|_{L^{2}}^{2}+\|v(t)\|_{L^{2}}^{2}+\|\theta(t)\|_{L^{2}}^{2}+ \int_{0}^{t}{
(\|\Lambda^{\alpha}
u(\tau)\|_{L^{2}}^{2}+\|\Lambda^{\beta} v(\tau)\|_{L^{2}}^{2})\,d\tau}\leq C(u_0, v_0, \theta_0).\nonumber
\end{align}
\end{lemma}

\begin{proof}
Taking the inner product of (\ref{casTCM}) with $(u, v, \theta)$ and using $\nabla\cdot u=0$, it follows from integration by parts that
\begin{eqnarray}\label{ttkl001}
\frac{1}{2}\frac{d}{dt}(\|u(t)\|_{L^{2}}^{2}+\|v(t)\|_{L^{2}}^{2}+\|\theta(t)
\|_{L^{2}}^{2})+\|\Lambda^{\alpha}
u\|_{L^{2}}^{2}+\|\Lambda^{\beta} v\|_{L^{2}}^{2}=0,
\end{eqnarray}
where the following cancellation identities have been used
$$\int_{\mathbb{R}^{2}}{\nabla\cdot(v\otimes v)\cdot u
\,dx}+\int_{\mathbb{R}^{2}}{ (v \cdot \nabla)u\cdot v\,dx}=0,$$
$$\int_{\mathbb{R}^{2}}{\nabla \theta\cdot v
\,dx}+\int_{\mathbb{R}^{2}}{(\nabla\cdot v)\theta\,dx}=0.$$
The desired estimate follows by integrating (\ref{ttkl001}) in the time. This completes the proof of Lemma \ref{cale21}.
\end{proof}

\vskip .1in
Next we establish the following crucial estimate.
\begin{lemma}\label{cale22}
Assume that $(u_{0},v_{0},\theta_{0})$ satisfies the conditions stated in Theorem \ref{Th1}. If $\alpha+\beta\geq2$ with $1<\alpha<2$,
then for any corresponding smooth solution $(u, v, \theta)$ of (\ref{casTCM}), we have, for any $t>0$,
\begin{align}\label{ttkl002}
\int_{0}^{t}{
\|\nabla u(\tau)\|_{L^{\infty}}\,d\tau}\leq C(t,u_0, v_0, \theta_0).
\end{align}
\end{lemma}

\vskip .1in
Before proving Lemma \ref{cale22}, we need the following space-time estimate of the solution of a linear equation with fractional diffusion. We point out that the Lebesgue space version of Lemma \ref{fraest01} previously appeared in \cite{DWWYez17}.
\begin{lemma}\label{fraest01}
Consider the following linear equation with $\gamma>0$,
\begin{eqnarray}\label{scvtyeq1}
\partial_{t}f+\Lambda^{\gamma}f=g,\qquad f(x,0)=f_{0}(x),\quad x\in \mathbb R^n,
\end{eqnarray}
then for any $1\leq p,\,r\leq\infty$ and for any $\sigma\in \mathbb{R}$, we have
\begin{eqnarray}\label{adscvtyeq2}
\|f\|_{L_{t}^{r}B_{p,r}^{\sigma}} \leq C(1+t)^{\frac{1}{r}}\|f_{0}\|_{B_{p,r}^{\sigma-{\frac{\gamma}{r}}}}+ C(1+t)\|g\|_{L_{t}^{r}B_{p,r}^{\sigma-\gamma}},
\end{eqnarray}
where $C=C(\sigma,\gamma)>0$ is a constant independent of variable $t$. In particular, there holds
\begin{eqnarray}\label{scvtyeq2}
\|f\|_{L_{t}^{1}B_{p,1}^{\sigma}} \leq C(1+t) \|f_{0}\|_{B_{p,1}^{\sigma-\gamma}}+ C(1+t)\|g\|_{L_{t}^{1}B_{p,1}^{\sigma-\gamma}},
\end{eqnarray}
\end{lemma}

\begin{proof}
Applying nonhomogeneous operator $\Delta_{j}$ (see Appendix \ref{aaa1} for its definition) to \eqref{scvtyeq1}, we have
\begin{eqnarray}
\partial_{t}\Delta_{j}f+\Lambda^{\gamma}\Delta_{j}f=\Delta_{j}g.\nonumber
\end{eqnarray}
By the Duhamel formula, we get
\begin{eqnarray}
\Delta_{j}f(t,x)=e^{-t\Lambda^{\gamma}}\Delta_{j}f_{0}(x)+\int_{0}^{t}
e^{-(t-\tau)\Lambda^{\gamma}}\Delta_{j}g(\tau,x)\,d\tau.\nonumber
\end{eqnarray}
For every $j\geq0$, we may conclude the following estimate
\begin{eqnarray}\label{cftghy1}
\|e^{-t\Lambda^{\gamma}}\Delta_{j}h\|_{L^{p}}\leq C_{1}e^{-C_{2}t2^{j\gamma}}\|\Delta_{j}h\|_{L^{p}},
\end{eqnarray}
where $C_{1}>0,\,C_{2}>0$ are two absolute constants independent of $j$. For the sake of completeness, the proof of \eqref{cftghy1} will be provided in Appendix \ref{aaa3}. Using \eqref{cftghy1}, it implies
\begin{align}
\|\Delta_{j}f\|_{L^{p}}\leq C_{1}e^{-C_{2}t2^{j\gamma}}\|\Delta_{j}f_{0}\|_{L^{p}}+C_{1}\int_{0}^{t}
e^{-C_{2}(t-\tau)2^{j\gamma}}
\|\Delta_{j}g(\tau,x)\|_{L^{p}}\,d\tau,\nonumber
\end{align}
which further gives
\begin{align}\label{cftghy2}
\|\Delta_{j}f\|_{L_{t}^{r}L^{p}}\leq C2^{-\frac{j\gamma}{r}} ( \|\Delta_{j}f_{0}\|_{L^{p}}+
2^{-j\gamma(1-\frac{1}{r})}\|\Delta_{j}g \|_{L_{t}^{r}L^{p}}).
\end{align}
Invoking the following estimate (see \cite{Wuejde})
\begin{align} 
\|e^{-t\Lambda^{\gamma}} h\|_{L^{p}}\leq C \| h\|_{L^{p}},\nonumber
\end{align}
we derive
\begin{align}\label{cftghy4}
\|\Delta_{-1}f\|_{L_{t}^{r}L^{p}}&\leq t^{\frac{1}{r}}\|\Delta_{-1}f\|_{L_{t}^{\infty}L^{p}}\nonumber\\&\leq Ct^{\frac{1}{r}}( \|\Delta_{-1}f_{0}\|_{L^{p}}+
\|\Delta_{-1}g \|_{L_{t}^{1}L^{p}})\nonumber\\&\leq Ct^{\frac{1}{r}}( \|\Delta_{-1}f_{0}\|_{L^{p}}+t^{1-\frac{1}{r}}
\|\Delta_{-1}g \|_{L_{t}^{r}L^{p}}).
\end{align}
Summing up \eqref{cftghy2} and \eqref{cftghy4} leads to
\begin{eqnarray}
\|f\|_{\widetilde{L}_{t}^{r}B_{p,r}^{\sigma}} \leq C(1+t)^{\frac{1}{r}}\|f_{0}\|_{B_{p,r}^{\sigma-{\frac{\gamma}{r}}}}+ C(1+t)\|g\|_{\widetilde{L}_{t}^{r}B_{p,r}^{\sigma-\gamma}}.\nonumber
\end{eqnarray}
Due to the fact $\widetilde{L}_{{t}}^{r}B_{p,r}^{s}\thickapprox {L}_{t}^{r}B_{p,r}^{s}$, we further get
\begin{eqnarray}
\|f\|_{ {L}_{t}^{r}B_{p,r}^{\sigma}} \leq C(1+t)^{\frac{1}{r}}\|f_{0}\|_{B_{p,r}^{\sigma-{\frac{\gamma}{r}}}}+ C(1+t)\|g\|_{ {L}_{t}^{r}B_{p,r}^{\sigma-\gamma}}.\nonumber
\end{eqnarray}
which is the desired estimate \eqref{adscvtyeq2}. Thus, we complete the proof of Lemma \ref{fraest01}.
\end{proof}

\vskip .1in
With Lemma \ref{fraest01} at our disposal, we are ready to prove Lemma \ref{cale22}.
\begin{proof}[{Proof of Lemma \ref{cale22}}]
Due to $\nabla\cdot u=0$, we rewrite $(\ref{casTCM})_{1}$ as
\begin{eqnarray}\label{ttkl004}
\partial_{t}u+\Lambda^{2\alpha} u=-\left(\mathbb{I}+(-\Delta)^{-1}\nabla\nabla\cdot\right)\nabla\cdot(v\otimes v+u\otimes u),\quad u(x,0)=u_{0}(x),
\end{eqnarray}
where $\mathbb{I}$ is the second order identity matrix.
Applying \eqref{scvtyeq2} to \eqref{ttkl004}, it follows that for any $p\in (1,\,\infty)$
\begin{align}\label{ttkl005}
\|u\|_{L_{t}^{1}B_{p,1}^{1+\frac{2}{p}}} \leq & C(1+t)\|u_{0}\|_{B_{p,1}^{1+\frac{2}{p}-2\alpha}}\nonumber\\&+ C(1+t)\|\left(\mathbb{I}+(-\Delta)^{-1}\nabla\nabla\cdot\right)\nabla\cdot(v\otimes v+u\otimes u)\|_{L_{t}^{1}B_{p,1}^{1+\frac{2}{p}-2\alpha}}\nonumber\\
\leq & C(t,u_{0}) + C(1+t)\|\nabla\cdot(v\otimes v+u\otimes u)\|_{L_{t}^{1}B_{p,1}^{1+\frac{2}{p}-2\alpha}}\nonumber\\
\leq & C(t,u_{0}) + C(1+t)\|v\otimes v+u\otimes u\|_{L_{t}^{1}B_{p,1}^{2+\frac{2}{p}-2\alpha}}\nonumber\\
\leq & C(t,u_{0}) + C(1+t)\|vv\|_{L_{t}^{1}B_{p,1}^{2+\frac{2}{p}-2\alpha}}
+ C(1+t)\|u u\|_{L_{t}^{1}B_{p,1}^{2+\frac{2}{p}-2\alpha}},
\end{align}
where we have used the boundedness of the Resiz type operator between the Besov spaces, namely,
$$\|(-\Delta)^{-1}\nabla\nabla h\|_{B_{p,q}^{s}}\leq C\| h\|_{B_{p,q}^{s}}$$
for any $p\in (1,\,\infty)$, $s\in \mathbb{R}$ and $q\geq1$. Now taking $r\in[1,\,\min\{p,\ \frac{1}{\alpha-1}\})$, we get by the embedding inequality \eqref{cvfEmd01} and the bilinear estimate \eqref{xcrty12} that
\begin{align}
C(1+t)\|vv\|_{L_{t}^{1}B_{p,1}^{2+\frac{2}{p}-2\alpha}} \leq& C(1+t)\|vv\|_{L_{t}^{1}B_{r,1}^{2+\frac{2}{r}-2\alpha}}\nonumber\\
\leq& C(1+t)\|v\|_{L_{t}^{2}B_{r_{1},2}^{-\delta}}\|v\|_{L_{t}^{2}
B_{r_{2},2}^{2+\frac{2}{r}-2\alpha+\delta}}
\nonumber\\
\leq& C(1+t)\|v\|_{L_{t}^{2}H^{\beta}}^{2},\nonumber
\end{align}
where $\delta>0$, $r_{1}\geq2$ and $r_{2}\geq2$ satisfy
$$\frac{1}{r_{1}}+\frac{1}{r_{2}}=\frac{1}{r},\quad-\delta-\frac{2}{r_{1}}\leq\beta-1,\quad 2+\frac{2}{r}-2\alpha+\delta-\frac{2}{r_{2}}\leq\beta-1$$
or
\begin{align}
1-\beta-\frac{2}{r_{1}}\leq \delta\leq 2\alpha+\beta-3-\frac{2}{r_{1}}.\nonumber
\end{align}
If the above $\delta$ would work, then the following restrictions should hold true $$\alpha+\beta\geq2,\ 1<\alpha<2;$$
$$1\geq\frac{1}{r}=\frac{1}{r_{1}}+\frac{1}{r_{2}}>\max\left\{\alpha-1,\ \frac{1}{p}\right\},\qquad \frac{1}{r_{1}}<\frac{2\alpha+\beta-3}{2}.$$
Direct computations show that due to $\alpha+\beta\geq2$ with $1<\alpha<2$, all the above parameters $\delta$, $r_{1}$, $r_{2}$ and $r$ can be fixed. For example, considering the case $\alpha+\beta=2$ with $1<\alpha<2$, we first take some $p>\frac{1}{\alpha-1}$, then we may choose $\delta$, $r_{1}$, $r_{2}$ and $r$ as
\begin{equation}
\begin{split}
& \delta=\frac{2\alpha-2-\max\{3\alpha-4,\,0\}}{4}; \\
& r_{1}=\frac{8}{2\alpha-2+\max\{3\alpha-4,\,0\}};\\
& r_{2}=\frac{4}{\alpha};\\
& r=\frac{8}{4\alpha-2+\max\{3\alpha-4,\,0\}}.\nonumber
\end{split}
\end{equation}
Therefore, we derive
\begin{align}\label{tdfty9}
C(1+t)\|vv\|_{L_{t}^{1}B_{p,1}^{2+\frac{2}{p}-2\alpha}}
\leq  C(1+t)\|v\|_{L_{t}^{2}H^{\beta}}^{2}.
\end{align}
Similarly, taking $\widetilde{r}\in[1,\,\min\{p,\ \frac{1}{\alpha-1}\})$, we have
\begin{align}
C(1+t)\|uu\|_{L_{t}^{1}B_{p,1}^{2+\frac{2}{p}-2\alpha}} \leq& C(1+t)\|uu\|_{L_{t}^{1}B_{\widetilde{r},1}^{2+\frac{2}{\widetilde{r}}-2\alpha}}\nonumber\\
\leq& C(1+t)\|u\|_{L_{t}^{2}B_{\widetilde{r}_{1},2}^{-\widetilde{\delta}}}\|u\|_{L_{t}^{2}
B_{r_{2},2}^{2+\frac{2}{\widetilde{r}}-2\alpha+\widetilde{\delta}}}
\nonumber\\
\leq& C(1+t)\|u\|_{L_{t}^{2}H^{\alpha}}^{2},\nonumber
\end{align}
where $\widetilde{\delta}>0$, $\widetilde{r}_{1}\geq2$ and $\widetilde{r}_{2}\geq2$ satisfy
$$\frac{1}{\widetilde{r}_{1}}+\frac{1}{\widetilde{r}_{2}}=\frac{1}{\widetilde{r}},
\quad-\widetilde{\delta}-\frac{2}{\widetilde{r}_{1}}\leq \alpha-1,\quad 2+\frac{2}{\widetilde{r}}-2\alpha+\widetilde{\delta}-\frac{2}{\widetilde{r}_{2}}\leq \alpha-1$$
or
\begin{align}
1-\alpha-\frac{2}{\widetilde{r}_{1}}\leq \widetilde{\delta}\leq 3\alpha-3-\frac{2}{\widetilde{r}_{1}}.\nonumber
\end{align}
Moreover, thanks to $1<\alpha<2$, we can also show that all the above parameters $\widetilde{\delta}$, $\widetilde{r}_{1}$, $\widetilde{r}_{2}$ and $\widetilde{r}$ can be fixed. In fact, we also take some $p>\frac{1}{\alpha-1}$, then we may choose $\widetilde{\delta}$, $\widetilde{r}_{1}$, $\widetilde{r}_{2}$ and $\widetilde{r}$ as follows
\begin{equation}
\begin{split}
& \widetilde{\delta}=\frac{6\alpha-6-\min\{3\alpha-3,\,1\}-\max\{2\alpha-3,\,0\}}{2}; \\
& \widetilde{r}_{1}=\frac{4}{\min\{3\alpha-3,\,1\}+\max\{2\alpha-3,\,0\}};\\
& \widetilde{r}_{2}=2;\\
& \widetilde{r}=\frac{4}{2+\min\{3\alpha-3,\,1\}+\max\{2\alpha-3,\,0\}}.\nonumber
\end{split}
\end{equation}
As a result, we obtain
\begin{align}\label{tdfty10}
C(1+t)\|uu\|_{L_{t}^{1}B_{p,1}^{2+\frac{2}{p}-2\alpha}}
\leq  C(1+t)\|u\|_{L_{t}^{2}H^{\alpha}}^{2}.
\end{align}
Putting \eqref{tdfty9} and \eqref{tdfty10} into \eqref{ttkl005} yields for $p\in (1,\,\infty)$ that
$$\|u\|_{L_{t}^{1}B_{p,1}^{1+\frac{2}{p}}}\leq C(t,u_{0})+ C(1+t)\|v\|_{L_{t}^{2}H^{\beta}}^{2}+C(1+t)\|u\|_{L_{t}^{2}H^{\alpha}}^{2}\leq C(t,u_{0},v_{0},\theta_{0}).$$
By the simple embedding inequality, we arrive at
\begin{align}
\int_{0}^{t}{
\|\nabla u(\tau)\|_{L^{\infty}}\,d\tau}\leq \int_{0}^{t}{
\|u(\tau)\|_{B_{p,1}^{1+\frac{2}{p}}}\,d\tau}\leq C(t,u_0, v_0, \theta_0).\nonumber
\end{align}
Thus, we complete the proof of Lemma \ref{cale22}.
\end{proof}

\vskip .1in
With \eqref{ttkl002} in hand, we are ready to show the global $H^1$-estimate.
\begin{lemma}\label{cale23}
Assume that $(u_{0},v_{0},\theta_{0})$ satisfies the conditions stated in Theorem \ref{Th1}. If $\alpha+\beta\geq2$ with $1<\alpha<2$,
then for any corresponding smooth solution $(u, v, \theta)$ of  (\ref{casTCM}), we have, for any $t>0$,
\begin{align}\label{ttkl007}
&\|\nabla u(t)\|_{L^{2}}^{2}+\|\nabla v(t)\|_{L^{2}}^{2}+\|\nabla\theta(t)\|_{L^{2}}^{2}+ \int_{0}^{t}{
(\|\Lambda^{\alpha}
\nabla u(\tau)\|_{L^{2}}^{2}+\|\Lambda^{\beta}\nabla v(\tau)\|_{L^{2}}^{2})\,d\tau}\nonumber\\&\leq C(t,u_0, v_0, \theta_0).
\end{align}
\end{lemma}

\begin{proof}
Taking the $L^{2}$ inner product of \eqref{casTCM} with $\Delta u$, $\Delta v$ and
$\Delta\theta$ respectively and adding them up, we get
\begin{align}\label{ttkl008}
&\frac{1}{2}\frac{d}{dt}(\|\nabla u(t)\|_{L^{2}}^{2}+\|\nabla v(t)\|_{L^{2}}^{2}
+\|\nabla\theta(t)\|_{L^{2}}^{2})+\|\Lambda^{\alpha}
\nabla u\|_{L^{2}}^{2}+\|\Lambda^{\beta}\nabla v\|_{L^{2}}^{2}\nonumber\\
 &=\int_{\mathbb{R}^{2}}{\Big(\nabla\cdot(v\otimes v)\cdot \Delta u
+ (v \cdot \nabla u)\cdot \Delta v\Big)\,dx}+\int_{\mathbb{R}^{2}}{ (u \cdot \nabla\theta)\cdot \Delta\theta\,dx}\nonumber\\& \quad
+\int_{\mathbb{R}^{2}}{ (u \cdot \nabla v)\cdot \Delta v\,dx}
\nonumber\\ &:=I_{1}+I_{2}+I_{3},
\end{align}
where we have used the identity
$$\int_{\mathbb{R}^{2}}{ (u \cdot \nabla u)\cdot \Delta u\,dx}=0.$$
Noticing $\alpha+\beta\geq2$ and using the Sobolev embedding inequalities, it yields
\begin{align}\label{ttkl009}
I_{1}&\leq C\int_{\mathbb{R}^{2}}{|v|\,|\nabla v|\,|\nabla^{2}u|\,dx} +C\int_{\mathbb{R}^{2}}{|\nabla v|\,|\nabla v|\,|\nabla u|\,dx}\nonumber\\&
\leq C\|v\|_{L^{\frac{2}{\alpha-1}}}\|\nabla v\|_{L^{2}}\|\nabla^{2}u\|_{L^{\frac{2}{2-\alpha}}}+C\|\nabla u\|_{L^{\infty}}\|\nabla v\|_{L^{2}}^{2}\nonumber\\&
\leq C\|v\|_{H^{\beta}}\|\nabla v\|_{L^{2}}\|\Lambda^{\alpha}\nabla u\|_{L^{2}}+C\|\nabla u\|_{L^{\infty}}\|\nabla v\|_{L^{2}}^{2}
\nonumber\\&
\leq \frac{1}{2}\|\Lambda^{\alpha}\nabla u\|_{L^{2}}^{2}+C\|v\|_{H^{\beta}}^{2}\|\nabla v\|_{L^{2}}^{2}+C\|\nabla u\|_{L^{\infty}}\|\nabla v\|_{L^{2}}^{2}.
\end{align}
Thanks to $\nabla\cdot u=0$, it is obvious to see that
\begin{align}\label{ttkl010}
I_{2}&\leq C\int_{\mathbb{R}^{2}}{|\nabla u|\,|\nabla \theta|^{2}\,dx} \leq C\|\nabla u\|_{L^{\infty}}\|\nabla \theta\|_{L^{2}}^{2},
\end{align}
\begin{align}\label{ttkl011}
I_{3}&\leq C\int_{\mathbb{R}^{2}}{|\nabla u|\,|\nabla v|^{2}\,dx} \leq C\|\nabla u\|_{L^{\infty}}\|\nabla v\|_{L^{2}}^{2}.
\end{align}
Substituting \eqref{ttkl009}, \eqref{ttkl010} and \eqref{ttkl011} into \eqref{ttkl008} gives
\begin{align*}
& \frac{d}{dt}(\|\nabla u(t)\|_{L^{2}}^{2}+\|\nabla v(t)\|_{L^{2}}^{2}
+\|\nabla\theta(t)\|_{L^{2}}^{2})+\|\Lambda^{\alpha}
\nabla u\|_{L^{2}}^{2}+\|\Lambda^{\beta}\nabla v\|_{L^{2}}^{2}\nonumber\\
 &\leq C(\|\nabla u\|_{L^{\infty}}+\|v\|_{H^{\beta}}^{2})(\|\nabla u\|_{L^{2}}^{2}+\|\nabla v\|_{L^{2}}^{2}
+\|\nabla\theta\|_{L^{2}}^{2}).
\end{align*}
Recalling \eqref{ttkl002}, we obtain the desired estimate \eqref{ttkl007} by using the Gronwall inequality. The proof of Lemma \ref{cale23} is thus completed.
\end{proof}

\vskip .1in
With the above estimates at our disposal, we are in a position to complete the proof of Theorem \ref{Th1}.
\begin{proof}[{Proof of Theorem \ref{Th1}}]
Applying $\Lambda^{s}$ with $s>2$ to the system (\ref{casTCM}),
taking the $L^{2}$ inner product with $\Lambda^{s}u$, $\Lambda^{s}v$ and
$\Lambda^{s}\theta$ respectively, and adding them up, we thus obtain
\begin{align}\label{ttkl012}
&\frac{1}{2}\frac{d}{dt}(\|\Lambda^{s}u(t)\|_{L^{2}}^{2}+\|\Lambda^{s}v(t)\|_{L^{2}}^{2}
+\|\Lambda^{s}\theta(t)\|_{L^{2}}^{2})+\|\Lambda^{s+\alpha}
u\|_{L^{2}}^{2}+\|\Lambda^{s+\beta}v\|_{L^{2}}^{2}\nonumber\\
&= -\int_{\mathbb{R}^{2}}{\Big(\Lambda^{s}\nabla\cdot(v\otimes v)\cdot \Lambda^{s} u
+\Lambda^{s}(v \cdot \nabla u)\cdot \Lambda^{s} v\Big)\,dx}-\int_{\mathbb{R}^{2}}{\Lambda^{s}(u \cdot \nabla\theta)\cdot \Lambda^{s} \theta\,dx}\nonumber\\& \quad
-\int_{\mathbb{R}^{2}}{ \Lambda^{s}(u \cdot \nabla v)\cdot \Lambda^{s} v\,dx}-\int_{\mathbb{R}^{2}}{ \Lambda^{s}(u \cdot \nabla u)\cdot \Lambda^{s}u\,dx}
\nonumber\\&:=J_{1}+J_{2}+J_{3}+J_{4},
\end{align}
where we have used the fact
$$\int_{\mathbb{R}^{2}}{\Lambda^{s}\nabla \theta\cdot \Lambda^{s}v
\,dx}+\int_{\mathbb{R}^{2}}{\Lambda^{s}(\nabla\cdot v)\Lambda^{s}\theta\,dx}=0.$$
In order to handle the terms at the right hand side of \eqref{ttkl012}, we will take advantage of the following commutator
estimates and bilinear estimates (see for example \cite{KP,Kenig})
\begin{eqnarray}\label{yfzdf68b1}
\|[\Lambda^{s},
f]g\|_{L^{p}}\leq C(\|\nabla
f\|_{L^{p_{1}}}\|\Lambda^{s-1}g\|_{L^{p_{2}}}+
\|\Lambda^{s}f\|_{L^{p_{3}}}\|g\|_{L^{p_{4}}}),
\end{eqnarray}
\begin{eqnarray} \label{yfzdf68b2}
\|\Lambda^{s}(f g)\|_{L^{p}}\leq C(\| f\|_{L^{p_{1}}}
\|\Lambda^{s}g\|_{L^{p_{2}}}+\|\Lambda^{s}f\|_{L^{p_{3}}}\|g\|_{L^{p_{4}}})
\end{eqnarray}
with $s>0,\,p_{2}, p_{3}\in(1, \infty)$ such that
$\frac{1}{p}=\frac{1}{p_{1}}+\frac{1}{p_{2}}=\frac{1}{p_{3}}+\frac{1}{p_{4}}$.
In some context, we also need the following variant version of (\ref{yfzdf68b1}), whose proof is the same one as for (\ref{yfzdf68b1})
\begin{eqnarray}\label{yfzdf68b3}
\|[\Lambda^{s-1}\partial_{i},f ]g \|_{L^r}
  \leq C\left(\|\nabla f \|_{L^{p_1}}\|\Lambda^{s-1} g \|_{L^{q_1}}
        +     \|\Lambda^{s} f \|_{L^{p_2}} \|  g \|_{L^{q_2}}\right).
        \end{eqnarray}
In view of \eqref{yfzdf68b1} and \eqref{yfzdf68b2}, it follows that
\begin{align}
J_{1} \leq& C \|\Lambda^{s}(v\otimes v)\|_{L^{2}}\|\Lambda^{s+1}u\|_{L^{2}}+C \|\Lambda^{s} (v\cdot\nabla u)\|_{L^{2}}\|\Lambda^{s}v\|_{L^{2}}\nonumber\\  \leq& C\|v\|_{L^{\infty}}\|\Lambda^{s}v\|_{L^{2}}\|\Lambda^{s+1}u\|_{L^{2}}
+C(\|\nabla u\|_{L^{\infty}}\|\Lambda^{s}v\|_{L^{2}}+\|\Lambda^{s}\nabla u\|_{L^{2}}\|v\|_{L^{\infty}})\|\Lambda^{s}v\|_{L^{2}}\nonumber\\ \leq& C\|v\|_{H^{\beta+1}}\|\Lambda^{s}v\|_{L^{2}}(\|\Lambda^{s}u\|_{L^{2}}
+\|\Lambda^{s+\alpha}u\|_{L^{2}})
+C\|\nabla u\|_{L^{\infty}}\|\Lambda^{s}v\|_{L^{2}}^{2}\nonumber\\ \leq& \frac{1}{8}\|\Lambda^{s+\alpha}u\|_{L^{2}}^{2}+C(1+\|v\|_{H^{\beta+1}}^{2}+\|\nabla u\|_{L^{\infty}})(\|\Lambda^{s}u\|_{L^{2}}^{2}+\|\Lambda^{s}v\|_{L^{2}}^{2}).\nonumber
\end{align}
By \eqref{yfzdf68b3}, we have
\begin{align}
J_{2} \leq& C\|[\Lambda^{s}\partial_{i}, u_{i}]\theta\|_{L^{2}}\|\Lambda^{s}\theta\|_{L^{2}}\nonumber\\
 \leq& C(\|\nabla u\|_{L^{\infty}}\|\Lambda^{s}\theta\|_{L^{2}}+\| \theta\|_{L^{\frac{2}{\alpha-1}}}\|\Lambda^{s+1} u\|_{L^{\frac{2}{2-\alpha}}})\|\Lambda^{s}\theta\|_{L^{2}}
\nonumber\\
 \leq& C\|\nabla u\|_{L^{\infty}}\|\Lambda^{s}\theta\|_{L^{2}}^{2}+C\| \theta\|_{L^{\frac{2}{\alpha-1}}}\|\Lambda^{s+\alpha} u\|_{L^{2}}\|\Lambda^{s}\theta\|_{L^{2}}
\nonumber\\
 \leq& \frac{1}{8}\|\Lambda^{s+\alpha}u\|_{L^{2}}^{2}+ C(\|\nabla u\|_{L^{\infty}}+\| \theta\|_{L^{\frac{2}{\alpha-1}}}^{2})\|\Lambda^{s}\theta\|_{L^{2}}^{2}.\nonumber
\end{align}
By the similar arguments, one derives
\begin{align}
J_{3} \leq& C\|[\Lambda^{s}\partial_{i}, u_{i}]v\|_{L^{2}}\|\Lambda^{s}v\|_{L^{2}}
 \nonumber\\
 \leq&
C(\|\nabla u\|_{L^{\infty}}\|\Lambda^{s}v\|_{L^{2}}+\| v\|_{L^{\frac{2}{\alpha-1}}}\|\Lambda^{s+1} u\|_{L^{\frac{2}{2-\alpha}}})\|\Lambda^{s}v\|_{L^{2}}
\nonumber\\
 \leq& C\|\nabla u\|_{L^{\infty}}\|\Lambda^{s}v\|_{L^{2}}^{2}+C\| v\|_{L^{\frac{2}{\alpha-1}}}\|\Lambda^{s+\alpha} u\|_{L^{2}}\|\Lambda^{s}v\|_{L^{2}}
\nonumber\\
 \leq& \frac{1}{8}\|\Lambda^{s+\alpha}u\|_{L^{2}}^{2}+ C(\|\nabla u\|_{L^{\infty}}+\| v\|_{L^{\frac{2}{\alpha-1}}}^{2})\|\Lambda^{s}v\|_{L^{2}}^{2},
 \nonumber
\end{align}
\begin{align}
J_{4}  \leq& C \int_{\mathbb{R}^{2}}{|[\Lambda^{s}, u\cdot\nabla]u\cdot
\Lambda^{s}u|\,dx}
\nonumber\\ \leq&
C\|\Lambda^{s}u\|_{L^{2}}\|[\Lambda^{s}, u\cdot\nabla]u\|_{L^{2}}
\nonumber\\ \leq&
C\|\nabla u\|_{L^{\infty}}\|\Lambda^{s}u\|_{L^{2}}^{2}.\nonumber
\end{align}
Substituting all the preceding estimates into (\ref{ttkl012}), one can finally get
\begin{align} &
\frac{d}{dt}(\|\Lambda^{s}u(t)\|_{L^{2}}^{2}+\|\Lambda^{s}v(t)\|_{L^{2}}^{2}
+\|\Lambda^{s}\theta(t)\|_{L^{2}}^{2})+\|\Lambda^{s+\alpha}
u\|_{L^{2}}^{2}+\|\Lambda^{s+\beta}v\|_{L^{2}}^{2}\nonumber\\
 &\leq C(1+\|v\|_{H^{\beta+1}}^{2}+\|\nabla u\|_{L^{\infty}}+\| \theta\|_{L^{\frac{2}{\alpha-1}}}^{2}+\| v\|_{L^{\frac{2}{\alpha-1}}}^{2})(\|\Lambda^{s}u\|_{L^{2}}^{2}+
\|\Lambda^{s}v\|_{L^{2}}^{2}+
\|\Lambda^{s}\theta\|_{L^{2}}^{2})\nonumber\\
 &\leq C(1+\|v\|_{H^{\beta+1}}^{2}+\| \theta\|_{H^{1}}^{2}+\|\nabla u\|_{L^{\infty}})(\|\Lambda^{s}u\|_{L^{2}}^{2}+
\|\Lambda^{s}v\|_{L^{2}}^{2}+
\|\Lambda^{s}\theta\|_{L^{2}}^{2}),\nonumber
\end{align}
which along with \eqref{ttkl002}, \eqref{ttkl007} and the Gronwall inequality yields
\begin{align} & \|\Lambda^{s}u(t)\|_{L^{2}}^{2}+\|\Lambda^{s}v(t)\|_{L^{2}}^{2}
+\|\Lambda^{s}\theta(t)\|_{L^{2}}^{2}\nonumber\\&\quad +\int_{0}^{t}{(\|\Lambda^{s+\alpha}
u(\tau)\|_{L^{2}}^{2}+\|\Lambda^{s+\beta}v(\tau)\|_{L^{2}}^{2})\,d\tau}<\infty.
\nonumber
\end{align}
Therefore, this completes the proof of Theorem \ref{Th1}.
\end{proof}

\vskip .3in
\section{The proof of Theorem \ref{Th2}}\setcounter{equation}{0}
This section is devoted to the proof of Theorem \ref{Th2}.
Similar to Lemma \ref{cale21}, we also have the basic global $L^2$-bound.
\begin{lemma}\label{yetle31}
Assume $(u_{0},v_{0},\theta_{0})$ satisfies the conditions stated in Theorem \ref{Th2}.
Then for any corresponding smooth solution $(u, v, \theta)$ of  (\ref{logcasTCM}), we have, for any $t>0$,
\begin{align}\label{tyht301}
\|u(t)\|_{L^{2}}^{2}+\|v(t)\|_{L^{2}}^{2}+\|\theta(t)\|_{L^{2}}^{2}+ \int_{0}^{t}{
 \|\mathcal{L}
u(\tau)\|_{L^{2}}^{2} \,d\tau}\leq C(u_0, v_0, \theta_0).
\end{align}
\end{lemma}

\vskip .1in
With the help of \eqref{tyht301}, we are able to derive the following key estimates.
\begin{lemma}\label{yetle32}
Assume $(u_{0},v_{0},\theta_{0})$ satisfies the conditions stated in Theorem \ref{Th2}. Then for any corresponding smooth solution $(u, v, \theta)$ of  (\ref{logcasTCM}), we have, for any $t>0$ and any $\sigma\in(0,\,1)$
\begin{eqnarray}\label{tyht302}
\|\Lambda^{\sigma}u(t)\|_{L^{2}}^{2}+\|\Lambda^{\sigma}v(t)\|_{L^{2}}^{2}
+\|\Lambda^{\sigma}\theta(t)\|_{L^{2}}^{2}+\int_{0}^{t}{\|\mathcal{L}\Lambda^{\sigma} u(\tau)\|_{L^{2}}^{2}\,d\tau}\leq C(t,u_0, v_0, \theta_0).
\end{eqnarray}
In particular, it holds
\begin{eqnarray}\label{tyht303}
\int_{0}^{t}{\|\nabla u(\tau)\|_{L^{\infty}}\,d\tau}\leq C(t,u_0, v_0, \theta_0).
\end{eqnarray}
\end{lemma}

\begin{proof}
Similar to \eqref{ttkl012}, we further have
\begin{align}\label{tyht304}
&\frac{1}{2}\frac{d}{dt}(\|\Lambda^{\sigma}u(t)\|_{L^{2}}^{2}+\|\Lambda^{\sigma}v(t)\|_{L^{2}}^{2}
+\|\Lambda^{\sigma}\theta(t)\|_{L^{2}}^{2})+\|\mathcal{L}\Lambda^{\sigma}
u\|_{L^{2}}^{2}\nonumber\\
&= -\int_{\mathbb{R}^{2}}{\Big(\Lambda^{\sigma}\nabla\cdot(v\otimes v)\cdot \Lambda^{\sigma} u
+\Lambda^{\sigma}(v \cdot \nabla u)\cdot \Lambda^{\sigma} v\Big)\,dx}-\int_{\mathbb{R}^{2}}{\Lambda^{\sigma}(u \cdot \nabla\theta)\cdot \Lambda^{\sigma} \theta\,dx}\nonumber\\& \quad
-\int_{\mathbb{R}^{2}}{ \Lambda^{\sigma}(u \cdot \nabla v)\cdot \Lambda^{\sigma} v\,dx}-\int_{\mathbb{R}^{2}}{ \Lambda^{\sigma}(u \cdot \nabla u)\cdot \Lambda^{\sigma}u\,dx}
\nonumber\\&:=\widetilde{J}_{1}+\widetilde{J}_{2}+\widetilde{J}_{3}+\widetilde{J}_{4}.
\end{align}
It follows from \eqref{yfzdf68b2} that
\begin{align}
\widetilde{J}_{1}&\leq C\|\Lambda^{\sigma} (vv)\|_{L^{\frac{2}{2-\sigma}}}\|\Lambda^{\sigma+1} u\|_{L^{\frac{2}{\sigma}}}
+C\|\Lambda^{\sigma}(v \cdot \nabla u)\|_{L^{2}}\|\Lambda^{\sigma} v\|_{L^{2}}\nonumber\\&\leq C\|v\|_{L^{\frac{2}{1-\sigma}}}\|\Lambda^{\sigma} v\|_{L^{2}}\|\Lambda^{\sigma+1} u\|_{L^{\frac{2}{\sigma}}}
+C\|\nabla u\|_{L^{\infty}}\|\Lambda^{\sigma}v\|_{L^{2}}\|\Lambda^{\sigma} v\|_{L^{2}}\nonumber\\&\leq C(\|\Delta u\|_{L^{2}}
+\|\nabla u\|_{L^{\infty}})\|\Lambda^{\sigma}v\|_{L^{2}}^{2}.\nonumber
\end{align}
Thanks to \eqref{yfzdf68b3}, we deduce
\begin{align}
\widetilde{J}_{2} \leq& C\|[\Lambda^{\sigma}\partial_{i}, u_{i}]\theta\|_{L^{2}}\|\Lambda^{\sigma}\theta\|_{L^{2}}\nonumber\\
 \leq& C(\|\nabla u\|_{L^{\infty}}\|\Lambda^{\sigma}\theta\|_{L^{2}}+\| \theta\|_{L^{\frac{2}{1-\sigma}}}\|\Lambda^{\sigma+1} u\|_{L^{\frac{2}{\sigma}}})\|\Lambda^{\sigma}\theta\|_{L^{2}}
\nonumber\\
 \leq& C(\|\Delta u\|_{L^{2}}
+\|\nabla u\|_{L^{\infty}})\|\Lambda^{\sigma}\theta\|_{L^{2}}^{2},\nonumber
\end{align}
\begin{align}
\widetilde{J}_{3} \leq& C\|[\Lambda^{\sigma}\partial_{i}, u_{i}]v\|_{L^{2}}\|\Lambda^{\sigma}v\|_{L^{2}}
 \nonumber\\
 \leq&
C(\|\nabla u\|_{L^{\infty}}\|\Lambda^{\sigma}v\|_{L^{2}}+\| v\|_{L^{\frac{2}{1-\sigma}}}\|\Lambda^{\sigma+1} u\|_{L^{\frac{2}{\sigma}}})\|\Lambda^{\sigma}v\|_{L^{2}}
\nonumber\\
 \leq& C(\|\Delta u\|_{L^{2}}
+\|\nabla u\|_{L^{\infty}})\|\Lambda^{\sigma}v\|_{L^{2}}^{2}.
 \nonumber
\end{align}
The last term can be easily bounded by
\begin{align}
\widetilde{J}_{4}  \leq&
C\|\Lambda^{\sigma}u\|_{L^{2}}\|[\Lambda^{\sigma}, u\cdot\nabla]u\|_{L^{2}}
\nonumber\\ \leq&
C\|\nabla u\|_{L^{\infty}}\|\Lambda^{\sigma}u\|_{L^{2}}^{2}.\nonumber
\end{align}
Putting the above estimates into \eqref{tyht304} gives
\begin{align}\label{tyht305}
& \frac{d}{dt}(\|\Lambda^{\sigma}u(t)\|_{L^{2}}^{2}+\|\Lambda^{\sigma}v(t)\|_{L^{2}}^{2}
+\|\Lambda^{\sigma}\theta(t)\|_{L^{2}}^{2})+\|\mathcal{L}\Lambda^{\sigma}
u\|_{L^{2}}^{2}
\nonumber\\&\leq
C(\|\Delta u\|_{L^{2}}
+\|\nabla u\|_{L^{\infty}})(\|\Lambda^{\sigma}u\|_{L^{2}}^{2}+\|\Lambda^{\sigma}v\|_{L^{2}}^{2}
+\|\Lambda^{\sigma}\theta\|_{L^{2}}^{2}).
\end{align}
Noticing the assumptions on $g$ (more precisely, $g$ grows logarithmically), we infer that for any fixed $\gamma>0$, there exists $N=N(\gamma)$ satisfying
$$g(r)\leq \widetilde{C}r^{\gamma},\quad \forall\,r\geq N$$
with the constant $\widetilde{C}=\widetilde{C}(\gamma)$.
Consequently, it follows that for any $0<\gamma<2$
\begin{align}\label{tyht306}
\|\mathcal{L}f\|_{L^{2}}^{2}&=
\int_{|\xi|< N(\gamma)}{ \frac{|\xi|^{4}}{g^{2}(|\xi|)}|\widehat{f}(\xi)|^{2}\,d\xi}
+\int_{|\xi|\geq N(\gamma)}{ \frac{|\xi|^{4}}{g^{2}(|\xi|)}|\widehat{f}(\xi)|^{2}\,d\xi}\nonumber\\
&\geq
\int_{|\xi|\geq N(\gamma)}{ \frac{|\xi|^{4}}{\big[\widetilde{C}|\xi|^{\gamma}
\big]^{2}}|\widehat{f}(\xi)|^{2}\,d\xi}\nonumber\\
&=
\int_{\mathbb{R}^{n}}{ \frac{|\xi|^{4}}{\big[\widetilde{C}|\xi|^{\gamma}
\big]^{2}}|\widehat{f}(\xi)|^{2}\,d\xi}-\int_{|\xi|<N(\gamma)}{ \frac{|\xi|^{4}}{\big[\widetilde{C}|\xi|^{\gamma}
\big]^{2}}|\widehat{f}(\xi)|^{2}\,d\xi}
\nonumber\\
&\geq  C_{1}\|\Lambda^{2-\gamma}f\|_{L^{2}}^{2}-{C_{2}}\|f\|_{L^{2}}^{2},
\end{align}
where $C_{1}$ and ${C_{2}}$ depend only on $\gamma$.
By the high-low frequency technique, we have
\begin{eqnarray}
\|\nabla u\|_{L^{\infty}}\leq \|\Delta_{-1}\nabla u\|_{L^{\infty}}+\sum_{l=0}^{N-1}\|\Delta_{l}\nabla u\|_{L^{\infty}}+\sum_{l=N}^{\infty}\|\Delta_{l}\nabla u\|_{L^{\infty}},\nonumber
\end{eqnarray}
where $\Delta_{l}$ ($l=-1,0,1,\cdot\cdot\cdot$) denote the frequency
operator (see Appendix for details).
By Lemma \ref{vfgty8xc}, one gets for $2-\sigma<\nu<2$
$$\|\Delta_{-1}\nabla u\|_{L^{\infty}}\leq C\|u\|_{L^{2}},$$
\begin{align}
\sum_{l=N}^{\infty}\|\Delta_{l}\nabla u\|_{L^{\infty}}&\leq
 C \sum_{l=N}^{\infty}2^{2l}\|\Delta_{l}u\|_{L^{2}}\nonumber\\&=
C \sum_{l=N}^{\infty}2^{l(2-\sigma-\nu)}
\|\Sigma_{l}\Lambda^{\sigma+\nu}u\|_{L^{2}}\nonumber\\&\leq  C2^{N(2-\sigma-\nu)}\|\Lambda^{\sigma+\nu}u\|_{L^{2}}\nonumber\\&\leq  C2^{N(2-\sigma-\nu)}(\|\Lambda^{\sigma}u\|_{L^{2}}+\|\mathcal{L} \Lambda^{\sigma} u\|_{L^{2}}),\nonumber
\end{align}
where in the last line we have used \eqref{tyht306}. From Lemma \ref{vfgty8xc} and the Plancherel theorem, we get
\begin{align}
\sum_{l=0}^{N-1}\|\Delta_{l}\nabla u\|_{L^{\infty}}&\leq C \sum_{l=0}^{N-1}2^{2l}\|\Delta_{l} u\|_{L^{2}}\leq C \sum_{l=0}^{N-1}\|\Delta_{l}\Delta u\|_{L^{2}}\nonumber\\
&\leq
C\sum_{l=0}^{N-1} \Big\|\varphi(2^{-l}\xi)
|\xi|^{2}\widehat{u}(\xi)\Big\|_{L^{2}}
\nonumber\\
&\leq  C\sum_{l=0}^{N-1} \Big\|\varphi(2^{-l}\xi)
g(|\xi|)
\frac{|\xi|^{2}}{ g(|\xi|) }\widehat{u}(\xi)\Big\|_{L^{2}}
\nonumber\\
&\leq  C\sum_{l=0}^{N-1} g(2^{l}) \Big\|\frac{|\xi|^{2}}{g(|\xi|) }\widehat{\Delta_{l}u}(\xi)\Big\|_{L^{2}}
\nonumber\\
&\leq C\Big(\sum_{l=0}^{N-1} g^{2}(2^{l})\Big)^{\frac{1}{2}}\left(\sum_{l=0}^{N-1}
\Big\|\frac{|\xi|^{2}}{ g (|\xi|)}\widehat{\Delta_{l}u}(\xi)\Big\|_{L^{2}}^{2}\right)^{\frac{1}{2}}
\nonumber\\
&\leq  Cg(2^{N})\Big(\sum_{l=1}^{N-1} 1\Big)^{\frac{1}{2}} \Big\|\frac{\Lambda^{2}}{ g(\Lambda)}u \Big\|_{L^{2}}
\nonumber\\
&\leq Cg (2^{N})\sqrt{N}\|\mathcal{L}u \|_{L^{2}}.\nonumber
\end{align}
As a result, we directly have
\begin{eqnarray}
\|\nabla u\|_{L^{\infty}}\leq C\|u\|_{L^{2}}+Cg (2^{N})\sqrt{N}\|\mathcal{L}u \|_{L^{2}}+C2^{N(2-\sigma-\nu)}(\|\Lambda^{\sigma}u\|_{L^{2}}+\|\mathcal{L} \Lambda^{\sigma} u\|_{L^{2}}).\nonumber
\end{eqnarray}
By the same arguments, we also obtain
\begin{eqnarray}
\|\Delta u\|_{L^{2}}\leq C\|u\|_{L^{2}}+Cg (2^{N})\sqrt{N}\|\mathcal{L}u \|_{L^{2}}+C2^{N(2-\sigma-\nu)}(\|\Lambda^{\sigma}u\|_{L^{2}}+\|\mathcal{L} \Lambda^{\sigma} u\|_{L^{2}}).\nonumber
\end{eqnarray}
For the sake of simplicity, we denote
$$X(t):=\|\Lambda^{\sigma}u(t)\|_{L^{2}}^{2}+\|\Lambda^{\sigma}v(t)\|_{L^{2}}^{2}
+\|\Lambda^{\sigma}\theta(t)\|_{L^{2}}^{2},\qquad Y(t):=\|\mathcal{L}\Lambda^{\sigma}
u(t)\|_{L^{2}}^{2}.$$
Then we deduce from \eqref{tyht305} that
$$\frac{d}{dt}X(t)+Y(t) \leq C\left(1+g (2^{N})\sqrt{N}\|\mathcal{L}u \|_{L^{2}}+2^{N(2-\sigma-\nu)}\big(X^{\frac{1}{2}}(t)+Y^{\frac{1}{2}}(t)\big)\right) X(t).$$
After choosing $N$ as
$$2^{N}\approx \left(e+X(t)\right)^{\frac{1}{\kappa}},\qquad \kappa:=2(\sigma+\nu-2)>0,$$
one concludes
\begin{align}
 \frac{d}{dt}X(t)+Y(t) &\leq C\left(1+g \left((e+X(t))^{\frac{1}{\kappa}}\right)\sqrt{\ln(e+X(t))}\|\mathcal{L}u \|_{L^{2}}\right) X(t) +CX^{\frac{1}{2}}(t)Y^{\frac{1}{2}}(t)\nonumber\\
 &\leq  \frac{1}{2}Y(t)+C\left(1+g \left((e+X(t))^{\frac{1}{\kappa}}\right)\sqrt{\ln(e+X(t))}\|\mathcal{L}u \|_{L^{2}}\right) X(t),\nonumber
\end{align}
which further leads to
\begin{eqnarray}\label{tyht307}
 \frac{d}{dt}X(t)+Y(t) \leq C(e+X(t))\sqrt{\ln(e+X(t))}g \left((e+X(t))^{\frac{1}{\kappa}}\right)(1+
 \|\mathcal{L}u \|_{L^{2}}).
\end{eqnarray}
It follows from (\ref{tyht307}) that
\begin{eqnarray}\label{tyht308}
\int_{e+X(0)}^{e+X(t)}\frac{d\tau}{\tau\sqrt{\ln \tau} g(\tau^{\frac{1}{\kappa}})}\leq C
\int_{0}^{t}{
\left(1+
 \|\mathcal{L}u(\tau)\|_{L^{2}}\right)\,d\tau}.
 \end{eqnarray}
It should be noted the following fact due to \eqref{logcobd}
\begin{eqnarray}\label{tyht309}
\int_{e}^{\infty}\frac{d\tau}{\tau\sqrt{\ln \tau} g(\tau^{\frac{1}{\kappa}})}=\sqrt{\kappa}\int_{e^{\frac{1}{\kappa}}}^{\infty}\frac{d\tau}{\tau\sqrt{\ln \tau} g(\tau)}=\infty.
\end{eqnarray}
Recalling \eqref{tyht301}, it yields
\begin{eqnarray}\label{tyht310}
\int_{0}^{t}{
\left(1+
 \|\mathcal{L}u(\tau)\|_{L^{2}}\right)\,d\tau}\leq C(t, u_{0},v_{0}, \theta_{0}).
 \end{eqnarray}
Keeping in mind \eqref{tyht308}, \eqref{tyht309} and \eqref{tyht310}, it is not hard to verify that $X(t)$ is finite for any given finite $t>0$, namely,
$$X(t)\leq C(t, u_{0},v_{0}, \theta_{0}).$$
Keeping in mind \eqref{tyht307}, we also get
$$\int_{0}^{t}{Y(\tau)\,d\tau}\leq C(t, u_{0},v_{0}, \theta_{0}).$$
As a result, the desired estimate \eqref{tyht302} follows immediately. By \eqref{tyht306}, we have
$$\|\Lambda^{\vartheta}u\|_{L^{2}}\leq C(\|u \|_{L^{2}}+\|\mathcal{L} \Lambda^{\sigma}u\|_{L^{2}}),\quad \forall\,\vartheta<\sigma+2.$$
Taking $\vartheta\in (2,\,\sigma+2)$ and using the simple interpolation yield
$$\int_{0}^{t}{\|\nabla u(\tau)\|_{L^{\infty}}\,d\tau}\leq C\int_{0}^{t}{(\| u(\tau)\|_{L^{2}}+\|\Lambda^{\vartheta}u(\tau)\|_{L^{2}})\,d\tau}\leq  C(t, u_{0},v_{0}, \theta_{0}),$$
which is \eqref{tyht303}. This completes the proof of Lemma \ref{yetle32}.
\end{proof}

\vskip .1in
With the above estimates in hand, we are ready to complete the proof of Theorem \ref{Th2}, which can be performed as that of Theorem \ref{Th1}. The details are as follows.
\begin{proof}[{Proof of Theorem \ref{Th2}}]
It follows from \eqref{ttkl012} that
\begin{align} \label{tyht311}
&\frac{1}{2}\frac{d}{dt}(\|\Lambda^{s}u(t)\|_{L^{2}}^{2}+\|\Lambda^{s}v(t)\|_{L^{2}}^{2}
+\|\Lambda^{s}\theta(t)\|_{L^{2}}^{2})+\|\mathcal{L}\Lambda^{s}
u\|_{L^{2}}^{2} \nonumber\\
&= -\int_{\mathbb{R}^{2}}{\Big(\Lambda^{s}\nabla\cdot(v\otimes v)\cdot \Lambda^{s} u
+\Lambda^{s}(v \cdot \nabla u)\cdot \Lambda^{s} v\Big)\,dx}-\int_{\mathbb{R}^{2}}{\Lambda^{s}(u \cdot \nabla\theta)\cdot \Lambda^{s} \theta\,dx}\nonumber\\& \quad
-\int_{\mathbb{R}^{2}}{ \Lambda^{s}(u \cdot \nabla v)\cdot \Lambda^{s} v\,dx}-\int_{\mathbb{R}^{2}}{ \Lambda^{s}(u \cdot \nabla u)\cdot \Lambda^{s}u\,dx}
\nonumber\\&:=J_{1}+J_{2}+J_{3}+J_{4}.
\end{align}
By \eqref{yfzdf68b1} and \eqref{yfzdf68b2}, we obtain
\begin{align}
J_{1} \leq& C \|\Lambda^{s}(v\otimes v)\|_{L^{\frac{2}{2-\sigma}}}\|\Lambda^{s+1}u\|_{L^{\frac{2}{\sigma}}}+C \|\Lambda^{s} (v\cdot\nabla u)\|_{L^{2}}\|\Lambda^{s}v\|_{L^{2}}\nonumber\\  \leq& C\|v\|_{L^{\frac{2}{1-\sigma}}}\|\Lambda^{s}v\|_{L^{2}}\|\Lambda^{s+2-\sigma}u\|_{L^{2}}
+C(\|\nabla u\|_{L^{\infty}}\|\Lambda^{s}v\|_{L^{2}}+\|\Lambda^{s}\nabla u\|_{L^{\frac{2}{\sigma}}}\|v\|_{L^{\frac{2}{1-\sigma}}})\|\Lambda^{s}v\|_{L^{2}}\nonumber\\ \leq& C\|\Lambda^{\sigma}v\|_{L^{2}}\|\Lambda^{s}v\|_{L^{2}}\|\Lambda^{s+2-\sigma}u\|_{L^{2}}
+C\|\nabla u\|_{L^{\infty}}\|\Lambda^{s}v\|_{L^{2}}^{2}.\nonumber
\end{align}
We also get by using \eqref{yfzdf68b3} that
\begin{align}
J_{2} \leq& C\|[\Lambda^{s}\partial_{i}, u_{i}]\theta\|_{L^{2}}\|\Lambda^{s}\theta\|_{L^{2}}\nonumber\\
 \leq& C(\|\nabla u\|_{L^{\infty}}\|\Lambda^{s}\theta\|_{L^{2}}+\| \theta\|_{L^{\frac{2}{1-\sigma}}}\|\Lambda^{s+1} u\|_{L^{\frac{2}{\sigma}}})\|\Lambda^{s}\theta\|_{L^{2}}
\nonumber\\
 \leq& C\|\nabla u\|_{L^{\infty}}\|\Lambda^{s}\theta\|_{L^{2}}^{2}+C\|\Lambda^{\sigma}\theta\|_{L^{2}}
 \|\Lambda^{s}\theta\|_{L^{2}}\|\Lambda^{s+2-\sigma}u\|_{L^{2}},\nonumber
\end{align}
\begin{align}
J_{3} \leq& C\|[\Lambda^{s}\partial_{i}, u_{i}]v\|_{L^{2}}\|\Lambda^{s}v\|_{L^{2}}
 \nonumber\\
 \leq&
C(\|\nabla u\|_{L^{\infty}}\|\Lambda^{s}v\|_{L^{2}}+\| v\|_{L^{\frac{2}{1-\sigma}}}\|\Lambda^{s+1} u\|_{L^{\frac{2}{\sigma}}})\|\Lambda^{s}v\|_{L^{2}}
\nonumber\\
 \leq& C\|\nabla u\|_{L^{\infty}}\|\Lambda^{s}v\|_{L^{2}}^{2}+C\|\Lambda^{\sigma}v\|_{L^{2}}
 \|\Lambda^{s}v\|_{L^{2}}\|\Lambda^{s+2-\sigma}u\|_{L^{2}}.
 \nonumber
\end{align}
The term $J_{4}$ admits the bound
\begin{align}
J_{4}  \leq
C\|\nabla u\|_{L^{\infty}}\|\Lambda^{s}u\|_{L^{2}}^{2}.\nonumber
\end{align}
Putting all the above estimates into \eqref{tyht311} yields
\begin{align} \label{tyht312}&
\frac{d}{dt}(\|\Lambda^{s}u(t)\|_{L^{2}}^{2}+\|\Lambda^{s}v(t)\|_{L^{2}}^{2}
+\|\Lambda^{s}\theta(t)\|_{L^{2}}^{2})+\|\mathcal{L}\Lambda^{s}
u\|_{L^{2}}^{2} \nonumber\\
 &\leq C\|\nabla u\|_{L^{\infty}}(\|\Lambda^{s}u\|_{L^{2}}^{2}+
\|\Lambda^{s}v\|_{L^{2}}^{2}+
\|\Lambda^{s}\theta\|_{L^{2}}^{2})\nonumber\\&\quad+C(\|\Lambda^{\sigma}v\|_{L^{2}}
+\|\Lambda^{\sigma}\theta\|_{L^{2}})
(\|\Lambda^{s}v\|_{L^{2}}+
\|\Lambda^{s}\theta\|_{L^{2}})\|\Lambda^{s+2-\sigma}u\|_{L^{2}}.
\end{align}
According to \eqref{tyht306}, we have
$$\|\Lambda^{s+2-\sigma}u\|_{L^{2}}\leq C(\|\Lambda^{s}u\|_{L^{2}}+\|\mathcal{L}\Lambda^{s}u\|_{L^{2}}),$$
which together with \eqref{tyht312} gives
\begin{align} \label{tyht313}&
\frac{d}{dt}(\|\Lambda^{s}u(t)\|_{L^{2}}^{2}+\|\Lambda^{s}v(t)\|_{L^{2}}^{2}
+\|\Lambda^{s}\theta(t)\|_{L^{2}}^{2})+\|\mathcal{L}\Lambda^{s}
u\|_{L^{2}}^{2} \nonumber\\
 &\leq C(1+\|\nabla u\|_{L^{\infty}}+\|\Lambda^{\sigma}v\|_{L^{2}}^{2}+
\|\Lambda^{\sigma}\theta\|_{L^{2}}^{2})(\|\Lambda^{s}u\|_{L^{2}}^{2}+
\|\Lambda^{s}v\|_{L^{2}}^{2}+
\|\Lambda^{s}\theta\|_{L^{2}}^{2}).
\end{align}
Noticing \eqref{tyht302}-\eqref{tyht303} and applying the Gronwall inequality to \eqref{tyht313}, we thus conclude
\begin{align}  \|\Lambda^{s}u(t)\|_{L^{2}}^{2}+\|\Lambda^{s}v(t)\|_{L^{2}}^{2}
+\|\Lambda^{s}\theta(t)\|_{L^{2}}^{2} +\int_{0}^{t}{\|\mathcal{L}\Lambda^{s}
u(\tau)\|_{L^{2}}^{2}\,d\tau}<\infty.
\nonumber
\end{align}
Consequently, we complete the proof of Theorem \ref{Th2}.
\end{proof}

\vskip .3in
\appendix
\section{Besov spaces and some useful facts}\label{aaa1}
This appendix recalls the Littlewood-Paley theory, introduces the
Besov spaces and provides Bernstein lemma.
We start with the Littlewood-Paley theory. We choose
some smooth radial non increasing function $\chi$ with values in $[0, 1]$ such that $\chi\in
C_{0}^{\infty}(\mathbb{R}^{n})$ is supported in the ball
$\mathcal{B}:=\{\xi\in \mathbb{R}^{n}, |\xi|\leq \frac{4}{3}\}$ and
and with value $1$ on $\{\xi\in \mathbb{R}^{n}, |\xi|\leq \frac{3}{4}\}$, then we set
$\varphi(\xi)=\chi\big(\frac{\xi}{2}\big)-\chi(\xi)$. One easily verifies that
${\varphi\in C_{0}^{\infty}(\mathbb{R}^{n})}$ is supported in the annulus
$\mathcal{C}:=\{\xi\in \mathbb{R}^{n}, \frac{3}{4}\leq |\xi|\leq
\frac{8}{3}\}$ and satisfies
$$\chi(\xi)+\sum_{j\geq0}\varphi(2^{-j}\xi)=1, \quad  \forall \xi\in \mathbb{R}^{n}.$$
Let $h=\mathcal{F}^{-1}(\varphi)$ and $\widetilde{h}=\mathcal{F}^{-1}(\chi)$, then we introduce the dyadic blocks $\Delta_{j}$ of our decomposition by setting
$$\Delta_{j}u=0,\ \ j\leq -2; \ \  \ \ \ \Delta_{-1}u=\chi(D)u=\int_{\mathbb{R}^{n}}{\widetilde{h}(y)u(x-y)\,dy};
$$
$$ \Delta_{j}u=\varphi(2^{-j}D)u=2^{jn}\int_{\mathbb{R}^{n}}{h(2^{j}y)u(x-y)\,dy},\ \ \forall j\in \mathbb{N}.
$$
We shall also use the
following low-frequency cut-off:
$$\ S_{j}u=\chi(2^{-j}D)u=\sum_{-1\leq k\leq j-1} \Delta_{k}u=2^{jn}\int_{\mathbb{R}^{n}}{\widetilde{h}(2^{j}y)u(x-y)\,dy},\ \ \forall j\in \mathbb{N}.$$

\vskip .1in
The nonhomogeneous Besov spaces are defined through
the dyadic decomposition.
\begin{define}
Let $s\in \mathbb{R}, (p,r)\in[1,+\infty]^{2}$. The nonhomogeneous
Besov space $B_{p,r}^{s}$ is defined as a space of $f\in
S'(\mathbb{R}^{n})$ such that
$$ B_{p,r}^{s}=\{f\in S'(\mathbb{R}^{n});  \|f\|_{B_{p,r}^{s}}<\infty\},$$
where
\begin{equation}\label{1}\nonumber
 \|f\|_{B_{p,r}^{s}}=\left\{\aligned
&\Big(\sum_{j\geq-1}2^{jrs}\|\Delta_{j}f\|_{L^{p}}^{r}\Big)^{\frac{1}{r}}, \quad \forall \ r<\infty,\\
&\sup_{j\geq-1}
2^{js}\|\Delta_{j}f\|_{L^{p}}, \quad \forall \ r=\infty.\\
\endaligned\right.
\end{equation}
\end{define}

\vskip .1in
Let us state the following classical facts
\begin{eqnarray}
\|f\|_{H^{s}}\approx \|f\|_{{B}_{2,\,2}^{s}},\qquad {B}_{p,\,r_{1}}^{s_{1}}\hookrightarrow {B}_{p,\,r_{2}}^{s_{2}},\,\,s_{1}>s_{2},\nonumber
\end{eqnarray}
\begin{eqnarray}\label{cvfEmd01}
{B}_{p_{1},\,r_{1}}^{s_{1}}\hookrightarrow {B}_{p_{2},\,r_{2}}^{s_{2}},\quad\,\,s_{1}-\frac{d}{p_{1}}=s_{2}-\frac{d}{p_{2}},\,\, 1\leq p_{1}\leq p_{2}\leq\infty,\,\,1\leq r_{1}\leq r_{2}\leq\infty.
\end{eqnarray}
We shall also need the mixed space-time spaces
$$ \|f\|_{L_{T}^{\rho}B_{p,r}^{s}}:=
\Big\|(2^{js}\|\Delta_{j}f\|_{L^{p}})_{l_{j}^{r}}\Big\|_{L_{T}^{\rho}}$$
and
$$ \|f\|_{\widetilde{L}_{T}^{\rho}B_{p,r}^{s}}:=
(2^{js}\|\Delta_{j}f\|_{L_{T}^{\rho}L^{p}})_{l_{j}^{r}}.$$
The following links are direct consequence of
the Minkowski inequality
$$L_{T}^{\rho}B_{p,r}^{s}\hookrightarrow \widetilde{L}_{T}^{\rho}B_{p,r}^{s},\qquad \mbox{if}\,\,r\geq \rho, \quad \mbox{and} \quad \widetilde{L}_{{T}}^{\rho}B_{p,r}^{s}\hookrightarrow {L}_{T}^{\rho}B_{p,r}^{s},\qquad \mbox{if}\,\,\rho\geq r.$$
In particular,
$$\widetilde{L}_{{T}}^{r}B_{p,r}^{s}\thickapprox {L}_{T}^{r}B_{p,r}^{s}.$$
\vskip .1in
We now introduce the Bernstein's inequalities, which are useful tools in dealing with Fourier localized functions and these
inequalities trade integrability for derivatives. The following lemma provides Bernstein
type inequalities for fractional derivatives

\begin{lemma} [see \cite{BCD}]\label{vfgty8xc}
Assume $1\leq a\leq b\leq\infty$. If the integer $j\geq-1$, then it holds
$$
\|\Lambda^{k}\Delta_{j}f\|_{L^b} \le C_1\, 2^{j k  +
jn(\frac{1}{a}-\frac{1}{b})} \|\Delta_{j}f\|_{L^a},\quad k\geq 0.
$$
If the integer $j\geq0$, then we have
$$
C_2\, 2^{ j k} \|\Delta_{j}f\|_{L^b } \le \|\Lambda^{k}\Delta_{j}f\|_{L^b } \le
C_3\, 2^{  j k + j n(\frac{1}{a}-\frac{1}{b})} \|\Delta_{j}f\|_{L^a},\quad
k\in \mathbb{{R}},
$$
where $C_1$, $C_2$ and $C_3$ are constants depending on $k,a$ and $b$
only.
\end{lemma}

\vskip .1in

Finally, we recall the following bilinear estimate in the nonhomogeneous Besov spaces, which can be proved by the same argument used in dealing with \cite[Lemma1]{Yzhangbo}.
\begin{lemma} 
Assume that $1\leq p,\,r\leq\infty,\, s>0,\,\delta_{1}>0,\,\delta_{2}>0$ and
 $1\leq p_{i},\, r_{i}\leq\infty \,(i=1,2,3,4)$ satisfy
$$\frac{1}{p}=\frac{1}{p_{1}}+\frac{1}{p_{2}}=\frac{1}{p_{3}}+\frac{1}{p_{4}},\qquad
\frac{1}{r}=\frac{1}{r_{1}}+\frac{1}{r_{2}}=\frac{1}{r_{3}}+\frac{1}{r_{4}}.$$
Then there exists a constant $C$ such that
\begin{eqnarray}
\|f g\|_{{B}_{p,r}^{s}}\leq C\|f \|_{{B}_{p_{1},r_{1}}^{-\delta_{1}}}\|g\|_{{B}_{p_{2},r_{2}}^{s+\delta_{1}}}
+C\|g\|_{{B}_{p_{3},r_{3}}^{-\delta_{2}}}
\|f\|_{{B}_{p_{4},r_{4}}^{s+\delta_{2}}}.\nonumber
\end{eqnarray}
In particular, it holds
\begin{eqnarray}\label{xcrty12}
\|ff\|_{{B}_{p,r}^{s}}\leq C\|f \|_{{B}_{p_{1},r_{1}}^{-\delta_{1}}}\|f\|_{{B}_{p_{2},r_{2}}^{s+\delta_{1}}}.
\end{eqnarray}
\end{lemma}

\vskip .3in

\section{The proof of \eqref{cftghy1}}\label{aaa3}
To show \eqref{cftghy1}, it suffices to show
\begin{eqnarray}\label{cftyezc01}
\mbox{Supp}\widehat{u}\subset \lambda \mathcal{C}\Rightarrow
\|e^{-t\Lambda^{\gamma}}u\|_{L^{p}}\leq C_{1}e^{-C_{2}t\lambda^{\gamma}}\|u\|_{L^{p}},
\end{eqnarray}
where $\lambda>0$ and $\mathcal{C}$ is an annulus. To this end, we consider a function $\phi$ in $\mathcal{S}(\mathbb{R}^{n}\backslash \{0\})$, the value of which is
identically $1$ near the annulus $\mathcal{C}$. We then have
\begin{align}
e^{-t\Lambda^{\gamma}}u&=\mathcal{F}^{-1}\left(e^{-t|\xi|^{\gamma}}
\widehat{u}\right)\nonumber\\
&=\mathcal{F}^{-1}\left(\phi(\lambda^{-1}\xi)e^{-t|\xi|^{\gamma}}
\widehat{u}\right)
\nonumber\\
&=g_{\lambda}(t,\cdot)\ast u,\nonumber
\end{align}
where
$$g_{\lambda}(t,x):=\int_{\mathbb{R}^{n}}e^{ix\cdot\xi}\phi(\lambda^{-1}\xi)e^{-t|\xi|^{\gamma}}
\,d\xi.$$
Direct computations yield
$$g_{\lambda}(t,x)=\lambda^{n}\widetilde{g}(\lambda^{\gamma}t,\lambda x),$$
where
$$\widetilde{g}(t,x):=\int_{\mathbb{R}^{n}}e^{ix\cdot\xi}\phi(\xi)e^{-t|\xi|^{\gamma}}
\,d\xi.$$
Thus, our goal is to find positive real numbers $C_{1}$ and $C_{2}$ such
that for any $t>0$
\begin{eqnarray}\label{cftyezc02}
\|\widetilde{g}(t,x)\|_{L^{1}}\leq C_{1}e^{-C_{2}t}.
\end{eqnarray}
It follows from integration by parts that
\begin{align}\label{cftyezc03}
\widetilde{g}(t,x)&=(1+|x|^{2})^{-n}\int_{\mathbb{R}^{n}}(1+|x|^{2})^{n}
e^{ix\cdot\xi}\phi(\xi)e^{-t|\xi|^{\gamma}}\,d\xi\nonumber\\
&=(1+|x|^{2})^{-n}\int_{\mathbb{R}^{n}}\left((\mathbb{I}_{n}-\Delta_{\xi})^{n}
e^{ix\cdot\xi}\right)\phi(\xi)e^{-t|\xi|^{\gamma}}\,d\xi\nonumber\\
&=(1+|x|^{2})^{-n}\int_{\mathbb{R}^{n}}
e^{ix\cdot\xi}(\mathbb{I}_{n}-\Delta_{\xi})^{n}\left(\phi(\xi)
e^{-t|\xi|^{\gamma}}\right)\,d\xi.
\end{align}
By the Leibniz formula
$$(\mathbb{I}_{n}-\Delta_{\xi})^{n}\left(\phi(\xi)
e^{-t|\xi|^{\gamma}}\right)=\sum_{\widetilde{\alpha}\leq\alpha,\ |\alpha|\leq 2n}C_{\alpha}^{\widetilde{\alpha}}\left(D^{\alpha-\widetilde{\alpha}}\phi(\xi)
\right)\left(D^{\widetilde{\alpha}}
e^{-t|\xi|^{\gamma}}\right)$$
Direct computations and the fact that the support
of $\phi$ is included in an annulus, it is not hard to show that
\begin{align}
\left|\left(D^{\alpha-\widetilde{\alpha}}\phi(\xi)
\right)\left(D^{\widetilde{\alpha}}
e^{-t|\xi|^{\gamma}}\right)\right|&\leq C(1+t)^{|\widetilde{\alpha}|}e^{-t|\xi|^{\gamma}}\nonumber\\&\leq C(1+t)^{|\widetilde{\alpha}|}e^{-2C_{2}t}\nonumber\\&\leq Ce^{-C_{2}t},\nonumber
\end{align}
which along with \eqref{cftyezc03} implies
$$|\widetilde{g}(t,x)|\leq C(1+|x|^{2})^{-n}e^{-C_{2}t}.$$
Consequently, we obtain the desired \eqref{cftyezc02}, namely,
\begin{eqnarray}
\|\widetilde{g}(t,x)\|_{L^{1}}\leq C_{1}e^{-C_{2}t}.\nonumber
\end{eqnarray}
This allows us to conclude
\begin{align}
\|e^{-t\Lambda^{\gamma}} u\|_{L^{p}}&=\|g_{\lambda}(t,\cdot)\ast u\|_{L^{p}} \nonumber\\&\leq  \|g_{\lambda}(t,x)\|_{L_{x}^{1}}\|u\|_{L^{p}}
\nonumber\\
&=\|\lambda^{n}\widetilde{g}(\lambda^{\gamma}t,\lambda x)\|_{L_{x}^{1}}\| u\|_{L^{p}}
\nonumber\\
&=\|\widetilde{g}(\lambda^{\gamma}t,x)\|_{L_{x}^{1}}\| u\|_{L^{p}}
\nonumber\\
&\leq C_{1}e^{-C_{2}t\lambda^{\gamma}}\| u\|_{L^{p}},\nonumber
\end{align}
which is the desired estimate \eqref{cftyezc01}.

\vskip .3in
\textbf{Acknowledgements.}
The author was supported by the National Natural Science Foundation of China (No. 11701232) and the Natural Science Foundation of Jiangsu
Province (No. BK20170224). This work was carried out when the author was visiting the Department of Mathematics, University of Pittsburgh.
The author appreciates the hospitality of Professor Dehua Wang and Professor Ming Chen.

\end{document}